\documentclass{amsart} 

\usepackage{amsmath,amssymb,amsthm} 
\usepackage{graphicx} 
\usepackage{subfig}
\usepackage[arrow, matrix, curve]{xy} 
\usepackage{xcolor} 
\usepackage{makecell} 
\usepackage{multirow} 
\usepackage{multicol}
\usepackage{enumerate}
\usepackage{booktabs} 
\usepackage{siunitx} 
\usepackage{enumitem} 
\usepackage{tikz}\usetikzlibrary{matrix,arrows}
\usepackage{url}
\usepackage{hyperref} 
\hypersetup{
	colorlinks,
	linkcolor={red!50!black},
	citecolor={blue!50!black},
	urlcolor={blue!80!black}
}  

\DeclareMathOperator{\sym}{Sym}

\newcommand{\R}{\mathbb{R}}
\newcommand{\Z}{\mathbb{Z}}
\newcommand{\Q}{\mathbb{Q}}

\newcommand{\spin}{\ifmmode{\rm Spin}\else{${\rm spin}$\ }\fi}
\newcommand{\spinc}{\ifmmode{{\rm Spin}^c}\else{${\rm spin}^c$}\fi}

\newcommand{\nbhd}{\mathcal{N}}

\newcommand{\cut}{\backslash}

\makeatletter
\newcommand{\Vast}{\bBigg@{2.5}} 
\makeatother

\makeatletter
\newtheorem*{rep@theorem}{\rep@title}
\newcommand{\newreptheorem}[2]{%
	\newenvironment{rep#1}[1]{%
		\def\rep@title{#2 \ref{##1}}%
		\begin{rep@theorem}}%
		{\end{rep@theorem}}}
\makeatother

\newtheoremstyle{thm}{}{}{\itshape}{}{\bfseries}{}{ }{} 
\newtheoremstyle{definition}{}{}{}{}{\bfseries}{}{ }{} 

\theoremstyle{thm}
\newtheorem{Theorem}{Theorem}[section]
\newtheorem{thm}[Theorem]{Theorem}
\newreptheorem{thm}{Theorem}  
\newtheorem{lem}[Theorem]{Lemma}
\newreptheorem{lem}{Lemma}  
\newtheorem{prop}[Theorem]{Proposition}

\newtheorem*{Theorem-ohne}{Theorem}
\newtheorem{con}[Theorem]{Conjecture}

\newtheorem{ques}[Theorem]{Question}

\newtheorem*{thm:counterexample}{Theorem~\ref{thm:counterexample}}
\newtheorem*{thm:torus_knotsQA}{Theorem~\ref{thm:torus_knotsQA}}
\newtheorem*{thm:QAsurgeries}{Theorem~\ref{thm:QAsurgeries}}
\newtheorem*{thm:asym_QA}{Theorem~\ref{thm:asym_QA}}
\newtheorem*{thm:exceptionalsym}{Theorem~\ref{thm:exceptionalsym}}
\newtheorem*{thm:2uncertain}{Theorem~\ref{thm:2uncertain}}
\newtheorem*{thm:BLknots}{Theorem~\ref{thm:BLknots}}

\theoremstyle{definition}

\newtheorem{rem}[Theorem]{Remark}
\newtheorem{notation}[Theorem]{Notation}

\begingroup\expandafter\expandafter\expandafter\endgroup
\expandafter\ifx\csname pdfsuppresswarningpagegroup\endcsname\relax
\else
\fi

\definecolor{amaranth}{rgb}{0.9, 0.17, 0.31} 
\definecolor{carrotorange}{rgb}{0.93, 0.57, 0.13} 
\definecolor{citrine}{rgb}{0.89, 0.82, 0.04} 
\definecolor{dartmouthgreen}{rgb}{0.05, 0.5, 0.06} 
\definecolor{ballblue}{rgb}{0.13, 0.67, 0.8} 
\definecolor{ceruleanblue}{rgb}{0.16, 0.32, 0.75} 
\definecolor{amethyst}{rgb}{0.6, 0.4, 0.8} 
\definecolor{amber}{rgb}{1.0, 0.75, 0.0} 
\definecolor{burlywood}{rgb}{0.87, 0.72, 0.53} 

\numberwithin{equation}{section}

\begin{document}
	
	
	\title{Two curious strongly invertible L-space knots} 
	
	\author{Kenneth L. Baker}
	\address{Department of Mathematics, University of Miami, Coral Gables, FL 33146, USA}
	\email{k.baker@math.miami.edu}
	
	\author{Marc Kegel}
    \address{Humboldt-Universit\"at zu Berlin, Rudower Chaussee 25, 12489 Berlin, Germany.}
    \address{Ruhr-Universit\"at Bochum, Universit\"atsstra{\ss}e 150, 44780 Bochum, Germany}
    \email{kegemarc@math.hu-berlin.de, kegelmarc87@gmail.com}
	
	\author{Duncan McCoy}
	\address{D\'{e}partment de Math\'{e}matiques, Universit\'{e} du Qu\'{e}bec \`{a} Montr\'{e}al, Canada}
	\email{mc\_coy.duncan@uqam.ca}
	
	
	\date{\today} 
	
\begin{abstract}
    We present two examples of strongly invertible L-space knots whose surgeries are never the double branched cover of a Khovanov thin link in the 3-sphere. Consequently, these knots provide counterexamples to a conjectural characterization of strongly invertible L-space knots due to Watson.
    We also discuss other exceptional properties of these two knots, for example, these two L-space knots have formal semigroups that are actual semigroups.
\end{abstract}

	
	\makeatletter
	\@namedef{subjclassname@2020}{%
		\textup{2020} Mathematics Subject Classification}
	\makeatother
	
	
	
	\maketitle
	
	
\section{Introduction}
A \textit{strong inversion} on a knot $K$ in $S^3$ is an element $\Phi \in \sym(S^3, K)$ defined by an orientation-preserving involution on $S^3$ that reverses the orientation of $K$. Using Khovanov homology,\footnote{In this paper, we work with reduced Khovanov homology groups with $\Z_2$-coefficients.} Watson introduced an invariant, $\varkappa(K,\Phi)$ of a knot $K$ equipped with a strong inversion $\Phi$ \cite{Watson_KH_sym2017}. The invariant $\varkappa(K, \Phi)$ is a finite-dimensional vector space with an absolute homological grading $h$ and a relative quantum $\Z$-grading $q$. Using this invariant Watson gave a conjectural characterization of strongly invertible L-space knots.\footnote{Here an \textit{L-space knot} is a knot $K$ that admits a surgery to a Heegaard Floer L-space.}

\begin{con}[{\cite[Conjecture~30]{Watson_KH_sym2017}}]\label{con:watson}
    A non-trivial knot $K$ admitting a strong inversion $\Phi$ is an L-space knot if and only if $\varkappa(K, \Phi)$ is supported in a single diagonal grading $\delta = q -2h$.
\end{con}

The main result of this article is a pair of counterexamples to one of the implications in this conjecture. Namely, we exhibit two L-space knots, $K_1$ and $K_2$ depicted in Figure~\ref{fig:example_knots}, for which $\varkappa$ is not thin.
 
\begin{thm}\label{thm:counterexample_Kappa}
    There exist two strongly invertible L-space knots, 
    $K_1$ and $K_2$ with unique strong inversions $\Phi_1$ and $\Phi_2$,
    such that $\varkappa(K_i, \Phi_i)$ is supported in two distinct $\delta$-gradings. In particular,\footnote{In this article, we adopt the convention that the relative $q$-grading and thus also the diagonal $\delta$-grading of $\varkappa(K,\Phi)$ are normalized by the Seifert longitude of $K$, see Section~\ref{sec:Kappa} for details. In the table presenting $\varkappa$ we record the non-vanishing dimensions of the graded vector spaces in $\varkappa(K,\Phi)$.}
   \[
\begin{tikzpicture}[scale=0.85]
\draw [black] (-2.2,0) -- (-2.2,1.5);
\draw [black] (-1,0) -- (-1,1.5);
\draw [gray] (-0.5,0) -- (-0.5,1.5);
\draw [gray] (0,0) -- (0,1.5);
\draw [gray] (0.5,0) -- (0.5,1.5);
\draw [gray] (1,0) -- (1,1.5);
\draw [gray] (1.5,0) -- (1.5,1.5);
\draw [gray] (2,0) -- (2,1.5);
\draw [gray] (2.5,0) -- (2.5,1.5);
\draw [gray] (3,0) -- (3,1.5);
\draw [gray] (3.5,0) -- (3.5,1.5);
\draw [black] (4,0) -- (4,1.5);

\draw [black] (-2.2,1.5) -- (4,1.5);
\draw [black] (-2.2,1) -- (4,1);
\draw [gray] (-2.2,0.5) -- (4,0.5);
\draw [black] (-2.2,0) -- (4,0);

\node at (-3.5,0.75) {$\varkappa(K_1,\Phi_1)=$};

\node at (-1.58,0.75) {\footnotesize{$\delta=17$}}; 
\node at (-1.58,0.25) {\footnotesize{$\delta=15$}}; 

\node at (-1.83,1.25) {\footnotesize{$h=$}};

\node at (-0.75,1.25) {\footnotesize{-$9$}}; 
\node at (-0.25,1.25) {\footnotesize{-$8$}}; 
\node at (0.25,1.25) {\footnotesize{-$7$}}; 
\node at (0.75,1.25) {\footnotesize{-$6$}}; 
\node at (1.25,1.25) {\footnotesize{-$5$}}; 
\node at (1.75,1.25) {\footnotesize{-$4$}}; 
\node at (2.25,1.25) {\footnotesize{-$3$}}; 
\node at (2.75,1.25) {\footnotesize{-$2$}}; 
\node at (3.25,1.25) {\footnotesize{-$1$}}; 
\node at (3.75,1.25) {\footnotesize{$0$}}; 

\node at (-0.75,.75) {$1$}; 
\node at (-0.25,.75) {$2$}; 
\node at (0.25,.75) {$3$}; 
\node at (0.75,.75) {$4$}; 
\node at (1.25,.75) {$4$}; 
\node at (1.75,.75) {$4$}; 
\node at (2.25,.75) {$3$}; 
\node at (2.75,.75) {$2$}; 
\node at (3.25,.75) {$1$}; 

\node at (1.25,.25) {$1$}; 
\node at (2.25,.25) {$1$}; 
\node at (2.75,.25) {$1$}; 
\node at (3.75,.25) {$1$}; 
\end{tikzpicture}\]
\[
\begin{tikzpicture}[scale=0.85]
\draw [black] (-2.2,0) -- (-2.2,1.5);
\draw [black] (-1,0) -- (-1,1.5);
\draw [gray] (-0.5,0) -- (-0.5,1.5);
\draw [gray] (0,0) -- (0,1.5);
\draw [gray] (0.5,0) -- (0.5,1.5);
\draw [gray] (1,0) -- (1,1.5);
\draw [gray] (1.5,0) -- (1.5,1.5);
\draw [gray] (2,0) -- (2,1.5);
\draw [gray] (2.5,0) -- (2.5,1.5);
\draw [gray] (3,0) -- (3,1.5);
\draw [gray] (3.5,0) -- (3.5,1.5);
\draw [gray] (4,0) -- (4,1.5);
\draw [black] (4.5,0) -- (4.5,1.5);

\draw [black] (-2.2,1.5) -- (4.5,1.5);
\draw [black] (-2.2,1) -- (4.5,1);
\draw [gray] (-2.2,0.5) -- (4.5,0.5);
\draw [black] (-2.2,0) -- (4.5,0);

\node at (-3.5,0.75) {$\varkappa(K_2,\Phi_2)=$};

\node at (-1.58,0.75) {\footnotesize{$\delta=17$}}; 
\node at (-1.58,0.25) {\footnotesize{$\delta=15$}}; 

\node at (-1.83,1.25) {\footnotesize{$h=$}};

\node at (-0.75,1.25) {\footnotesize{-$6$}}; 
\node at (-0.25,1.25) {\footnotesize{-$5$}}; 
\node at (0.25,1.25) {\footnotesize{-$4$}}; 
\node at (0.75,1.25) {\footnotesize{-$3$}}; 
\node at (1.25,1.25) {\footnotesize{-$2$}}; 
\node at (1.75,1.25) {\footnotesize{-$1$}}; 
\node at (2.25,1.25) {\footnotesize{$0$}}; 
\node at (2.75,1.25) {\footnotesize{$1$}}; 
\node at (3.25,1.25) {\footnotesize{$2$}}; 
\node at (3.75,1.25) {\footnotesize{$3$}}; 
\node at (4.25,1.25) {\footnotesize{$4$}}; 

\node at (-0.75,.75) {$1$}; 
\node at (-0.25,.75) {$2$}; 
\node at (0.25,.75) {$3$}; 
\node at (0.75,.75) {$4$}; 
\node at (1.25,.75) {$4$}; 
\node at (1.75,.75) {$4$}; 
\node at (2.25,.75) {$3$}; 
\node at (2.75,.75) {$2$}; 
\node at (3.25,.75) {$1$}; 

\node at (1.25,.25) {$1$};
\node at (1.75,.25) {$1$}; 
\node at (2.25,.25) {$1$}; 
\node at (2.75,.25) {$2$}; 
\node at (3.25,.25) {$1$}; 
\node at (3.75,.25) {$1$}; 
\node at (4.25,.25) {$1$}; 
\end{tikzpicture}\]
\end{thm}

\begin{figure}[htbp]
    \centering
    \includegraphics[width=0.42\textwidth]{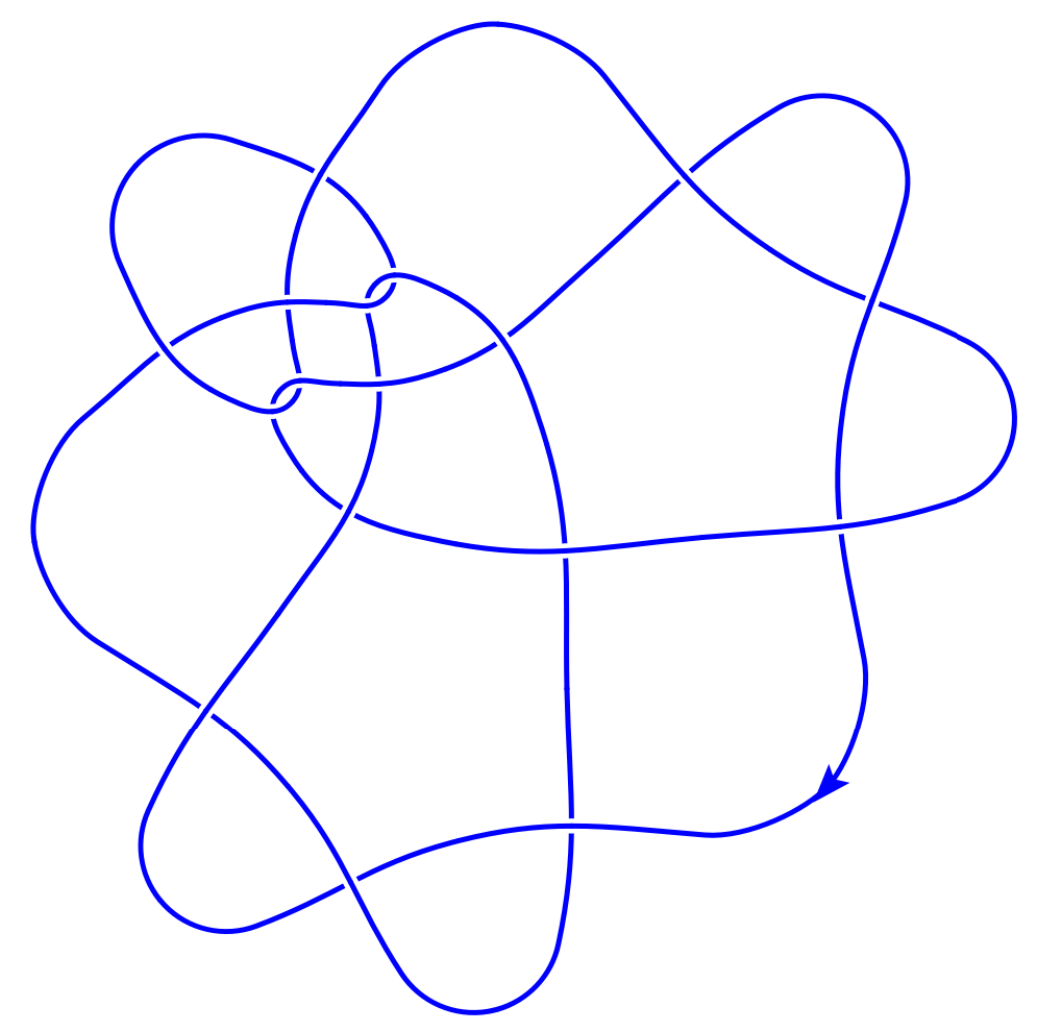}
    \includegraphics[width=0.56\textwidth]{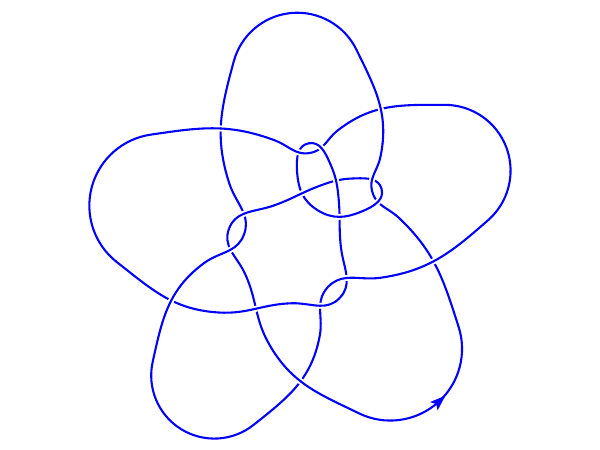}
    \caption{Strongly invertible diagrams for $K_1$ (left) and $K_2$ (right).}
    \label{fig:example_knots}
\end{figure}

In fact, these knots $K_1$ and $K_2$ turn out to have a stronger property. We say that a surgery on a knot is (Khovanov) \textit{thin} if it is the double branched cover of a knot or link with thin Khovanov homology. Conjecture~\ref{con:watson} implies that if $K$ is a strongly invertible L-space knot, then every sufficiently large $r\in \Q$ would be a thin surgery slope. The knots $K_1$ and $K_2$, however, do not admit any thin surgeries.
	
\begin{thm}\label{thm:counterexample_thin}
For all $r\in \Q$, performing $r$-surgery on $K_1$ or $K_2$ never yields the double branched cover of a knot or link with thin Khovanov homology.
\end{thm}
	
The knots $K_1$ and $K_2$ admit several descriptions. In Burton's notation they are $K_1=17nh0000014$ and $K_2=17nh0000019$, indicating they both have crossing number $17$, are hyperbolic, and non-alternating~\cite{Burton}. 
The complements of $K_1$ and $K_2$ also appear as manifolds in the SnapPy census~\cite{Du18}. The complements of $m K_1$ and $ m K_2$ are homeomorphic to $t09847$ and $o9\_30634$, respectively.
As braid closures, $K_1$ and $K_2$ are presented by 
\begin{align*}
K_1 &=[(2,1,3,2)^3,1,2,3,3,2] \\
K_2&=[(2,1,3,2)^3,-1,2,1,1,2].
\end{align*}
The knots $K_1$ and $K_2$ were identified as counter-examples to Conjecture~\ref{con:watson} following a systematic study of Dunfield's SnapPy census knots~\cite{Du18}, i.e.\ the hyperbolic knots in $S^3$ whose complements admit ideal triangulations of at most $9$ tetrahedra. It was shown by Dunfield~\cite{Du19}, cf.~\cite{BKM_QA} that approximately half of the census knots are L-space knots, making the SnapPy census a rich source of potentially interesting L-space knots. After an exploration of surgeries on the census L-space knots, we found that $K_1$ and $K_2$ were the only census L-space knots that did not appear to have any alternating or quasi-alternating surgeries \cite{BKM_alt, BKM_QA}, leaving them as the only possible counterexamples to Conjecture~\ref{con:watson} amongst the SnapPy census knots.

Coincidentally, $K_2$ is the only hyperbolic L-space knot that is known to not be braid positive~\cite{BakerKegel}.
Another interesting observation is that $K_1$ and $K_2$ are also the only two L-space knots in the SnapPy census whose formal semi-groups are actual semi-groups.  All other currently known hyperbolic L-space knots whose formal semi-groups are semi-groups are contained in the infinite families from~\cite{BakerKegel,Teragaito}. (The notion of a formal semi-group of an L-space knot was introduced in~\cite{formal_semi_group}. We refer to that article for the precise definition.)  Adapting Watson's conjecture, we ask: 
\begin{ques}
    Let $K$ be a strongly invertible hyperbolic L-space knot. Is every sufficiently large slope of $K$ a thin slope if and only if the formal semi-group of $K$ is not a semi-group?
\end{ques}

By Theorem~\ref{thm:counterexample_Kappa}, the knots $K_1$ and $K_2$ have $\varkappa(K_i, \Phi_i)$ of width two. It is natural to wonder if this can be improved upon.

\begin{ques}
Given $n\geq 3$, does there exist a strongly invertible L-space knot $(K,\Phi)$, such that $\varkappa(K,\Phi)$ has width at least $n$?
\end{ques}

The $\varkappa$-invariant of a strongly invertible knot $K$ with a unique strong inversion $\Phi$ is determined by the Khovanov homology of the integer tangle fillings $T(n)$ of the tangle exterior $T$ of $K$. On the other hand, it is known that there are only finitely many slopes for $K$ such that $K(r)$ is the double branched cover of a link $L\neq T(r)$. Any such slope is necessarily an exceptional or symmetry-exceptional slope of $K$ (see the proof of Theorem~\ref{thm:counterexample_thin}). One might wonder if such an \textit{exceptional} link $L$ might be thin without the corresponding tangle filling $T(r)$ being thin. Lemma~\ref{lem:sfs_branching} below shows that this cannot be the case if $K(r)$ is a Seifert fibered space.

\begin{ques}
    Does there exist a strongly invertible L-space knot $(K,\Phi)$ with a thin surgery such that $\varkappa(K,\Phi)$ is not thin?
\end{ques}

Note also that one direction of Conjecture~\ref{con:watson} remains open.
\begin{con}
    Let $K$ be a knot admitting a strong inversion $\Phi$ such that $\varkappa(K, \Phi)$ is supported in a single diagonal grading $\delta = q -2h$. Then $K$ is an L-space knot.
\end{con}
	
\subsection*{Code and data}
The code and additional data accompanying this paper can be accessed at~\cite{BKM}.

\subsection*{Acknowledgments}
We thank Dave Futer, Marco Golla, Liam Watson, and Clau\-di\-us Zibrowius for useful discussions, explanations, and remarks. For the computations performed in this paper, we use and combine KLO (Knot like objects)~\cite{Sw}, KnotJob~\cite{KnotJob}, the Mathematica knot theory package~\cite{knotatlas}, Regina~\cite{BBP+}, sage~\cite{sagemath}, SnapPy \cite{CDGW}, and code developed by Dunfield in~\cite{Du18,Du19}. We thank the creators for making these programs publicly available. 

Special thanks go to Claudius Zibrowius who verified our computations of the $\varkappa$-invariant of $K_1$ and $K_2$ using his program \textit{kht++}\cite{khtpp}. These calculations can be accessed at~\cite{khtpp}. This check revealed a mistake in the diagram and braid words presenting $K_1$ given in a previous version of this article.

We are also grateful to the anonymous referee for their helpful comments and thoughtful feedback on this article.

KLB was partially supported by the Simons Foundation grant \#523883 and gift \#962034.
 
MK was supported by the SFB/TRR 191 \textit{Symplectic Structures in Geometry, Algebra and Dynamics}, funded by the DFG (Projektnummer 281071066 - TRR 191).

DM is supported by an NSERC Discovery Grant and a Canada Research Chair.
	
\section{Khovanov homology and the \texorpdfstring{$\varkappa$}{K}-invariant.}\label{sec:Kappa}
In this paper, we work with reduced Khovanov homology $ {Kh}^{h,q}(L)$ with $\Z_2$-coefficients, where $h$ and $q$ denote the homological and quantum gradings, respectively. We also write ${Kh}(L)$ for $\oplus_{h,q}{Kh}^{h,q}(L)$. It is also convenient to define the diagonal $\delta$-grading $\delta=q-2h$. The (Khovanov homology) \textit{width} of a link $L$ is
\[
w_{Kh}(L)=\frac{1}{2}(\delta_{\rm max}-\delta_{\rm min}) + 1
\]
where $\delta_{\rm max}$ and $\delta_{\rm min}$ are the maximal and minimal $\delta$-gradings appearing in the support of $Kh(L)$, respectively.
A link is called \textit{thin} if its width is $1$. A link is \textit{thick} if it is not thin. While the Khovanov homology of a link may depend on the orientation of its components, the Khovanov width of a link is independent of the orientations of the link components. This follows, for example, from \cite[Proposition~28]{Khovanov_homology}. Thus, it makes sense to discuss the width for unoriented links.

Given a strongly invertible knot $(K,\Phi)$, the fixed points of $\Phi$ form an unknot in $S^3$. If one excises an open $\Phi$-equivariant tubular neighborhood of $K$ from $S^3$ and quotients by $\Phi$, then the fixed point locus of $\Phi$ becomes a properly embedded tangle $(T, B^3)$, the \textit{tangle exterior} of $K$. Taking the double branched cover of $B^3$ along $T$ yields the exterior of $K$. One takes a fixed diagram for $T$ (which we also denote by $T$) in $B^3$ such that that the $\infty$-tangle filling $T(\infty)$ of $T$ is unknotted and the $0$-tangle filling $T(0)$ is a 2-component link of determinant $0$ (see Figure~\ref{fig:tangle_conventions}). In general, these conventions ensure that the double branched cover of the link $T(p/q)$ will be diffeomorphic to $p/q$-surgery on $K$ (cf. the Montesinos trick~\cite{Montesinos_trick}).

\begin{figure}[htbp] 
	\centering
	\def\svgwidth{0,7\columnwidth}
\begingroup%
  \makeatletter%
  \providecommand\color[2][]{%
    \errmessage{(Inkscape) Color is used for the text in Inkscape, but the package 'color.sty' is not loaded}%
    \renewcommand\color[2][]{}%
  }%
  \providecommand\transparent[1]{%
    \errmessage{(Inkscape) Transparency is used (non-zero) for the text in Inkscape, but the package 'transparent.sty' is not loaded}%
    \renewcommand\transparent[1]{}%
  }%
  \providecommand\rotatebox[2]{#2}%
  \newcommand*\fsize{\dimexpr\f@size pt\relax}%
  \newcommand*\lineheight[1]{\fontsize{\fsize}{#1\fsize}\selectfont}%
  \ifx\svgwidth\undefined%
    \setlength{\unitlength}{530.89915617bp}%
    \ifx\svgscale\undefined%
      \relax%
    \else%
      \setlength{\unitlength}{\unitlength * \real{\svgscale}}%
    \fi%
  \else%
    \setlength{\unitlength}{\svgwidth}%
  \fi%
  \global\let\svgwidth\undefined%
  \global\let\svgscale\undefined%
  \makeatother%
  \begin{picture}(1,0.38741684)%
    \lineheight{1}%
    \setlength\tabcolsep{0pt}%
    \put(0,0){\includegraphics[width=\unitlength,page=1]{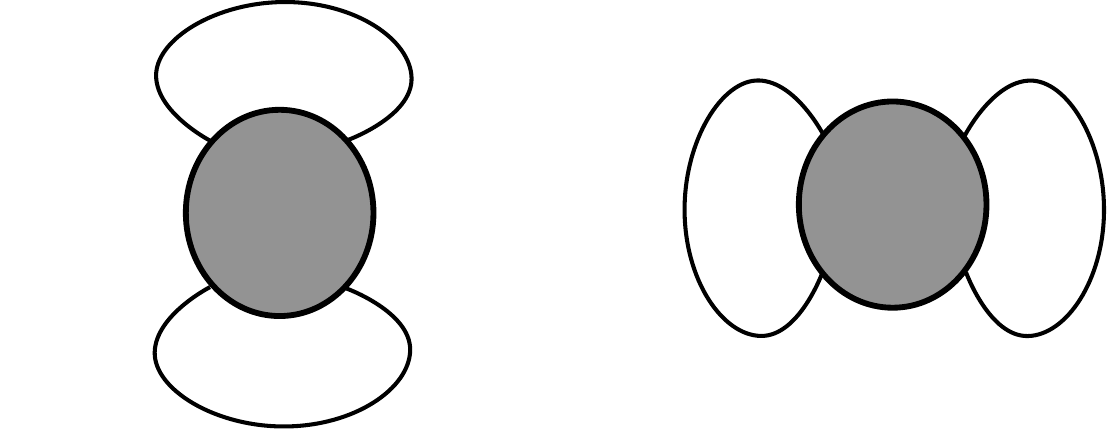}}%
    \put(0.23182785,0.18668545){\color[rgb]{0,0,0}\makebox(0,0)[lt]{\lineheight{1.25}\smash{\begin{tabular}[t]{l}$T$\end{tabular}}}}%
    \put(0.0006421,0.19677417){\color[rgb]{0,0,0}\makebox(0,0)[lt]{\lineheight{1.25}\smash{\begin{tabular}[t]{l}$T(0)=$\end{tabular}}}}%
    \put(0.43464296,0.20210227){\color[rgb]{0,0,0}\makebox(0,0)[lt]{\lineheight{1.25}\smash{\begin{tabular}[t]{l}$T(\infty)=$\end{tabular}}}}%
    \put(0.78609976,0.19439849){\color[rgb]{0,0,0}\makebox(0,0)[lt]{\lineheight{1.25}\smash{\begin{tabular}[t]{l}$T$\end{tabular}}}}%
  \end{picture}%
\endgroup%

	\caption{Left: The tangle filling of $T$ with slope $0$ yields a $2$-component link with determinant $0$ whose double branched cover is $K(0)$. Right: The tangle filling of $T$ with slope $\infty$ yields an unknot whose double branched cover is $S^3=K(\infty)$.}
	\label{fig:tangle_conventions}
\end{figure}
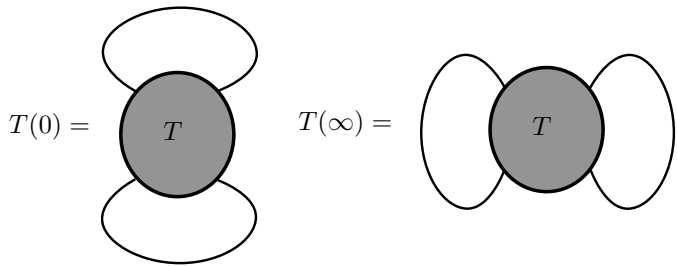

Since $T(n-1)$ can be obtained from $T(n)$ by resolving a single crossing, the long exact sequence in Khovanov homology contains a map
\[
f_n: {Kh}\big(T(n)\big)\rightarrow {Kh}\big(T(n-1)\big).
\]
This mapping preserves the homological grading but decreases the quantum grading by -1. Let $\mathbb{A}$ be the vector space obtained as the inverse limit of the system $({Kh}(T(n)),f_n)_{n\in \Z}$. That is, $\mathbb{A}$ consists of sequences $(x_n)_{n\in \Z}$ such that $x_n\in{Kh}(T(n))$ and $x_{n-1}=f_n(x_n)$ for all $n$. Let $\mathbb{K}$ be the subspace of $\mathbb{A}$ consisting of sequences such that $x_n=0$ for all $n$ sufficiently small. The invariant $\varkappa$ is defined to be the quotient
\[
\varkappa(K,\Phi)=\mathbb{A}/\mathbb{K}.
\]
Since the maps $f_n$ preserve the homological grading in Khovanov homology, this induces an absolute $\Z$-grading $h$ on $\varkappa$. Furthermore, since the $f_n$ preserve relative (but not absolute) $q$-gradings, $\varkappa$ also inherits a relative $q$-grading and hence a relative $\delta$-grading. Thus, it makes sense to speak of the width of $\varkappa$. In this article, we normalize the $q$- and $\delta$-gradings of $\varkappa$ so that they agree with the $q$- and $\delta$-gradings of $Kh(T(0))$.

In practice, the invariant $\varkappa$ is relatively easy to compute. The map $f_n$ sits inside the following exact triangle of homology groups:
\[\begin{tikzpicture}[>=latex] 
\matrix (m) [matrix of math nodes, row sep=1em,column sep=1em]
{ Kh(T(n)) & & Kh(T(n-1)) \\
& Kh(U)& \\ };
\path[->,font=\scriptsize]
(m-1-1) edge node[auto] {$ f_n $} (m-1-3)
(m-2-2.north west) edge (m-1-1.south east)
(m-1-3.south west) edge (m-2-2.north east);
\end{tikzpicture}.\]
Since the Khovanov homology $Kh(U)$ of the unknot $U$ has dimension one, it follows that for each $n$ the map $f_n$ is either (i) injective with cokernel of dimension one or (ii) surjective with kernel of dimension one. This allows the image of any $f_n$ to be easily calculated by comparing $Kh(T(n))$ and $Kh(T(n-1))$.

Moreover, Watson shows that for every strongly invertible knot $(K,\Phi)$ there is an integer $N$ such that $f_n$ is surjective for all $n>N$ and injective for all $n<N$ \cite{Watson2012}. Given this integer $N$, $\varkappa(K,\Phi)$ is then naturally isomorphic to the image of the map
\[
f_{N}\circ f_{N+1}: Kh(T(N+1))\rightarrow Kh(T(N-1)). 
\]
			
\section{Counterexamples to Watson's conjecture} \label{sec:QAKhovanov}
In this section, we prove Theorem~\ref{thm:counterexample_Kappa}. That is, we show that the two knots $K_1$ and $K_2$ provide counterexamples to Conjecture~\ref{con:watson}.
					
We write $K$ for either $K_1$ or $K_2$. From Figure~\ref{fig:example_knots} it is apparent that $K$ is strongly invertible. Furthermore, computing the symmetry group in SnapPy shows that the strong inversion on $K$ is unique. This allows us to drop the strong inversion from the notation. Hence there is a unique tangle exterior $T$ whose double branched cover is the exterior of $K$. These tangle exteriors are depicted in Figure~\ref{fig:tangle_exteriors}. For any slope $r\in \Q$ the manifold $K(r)$ is homeomorphic to the double branched cover of $T(r)$, the filling of $T$ by the rational tangle of slope $r$. 
				
\begin{figure}[htbp]
	\centering
    \includegraphics[width=.9\textwidth]{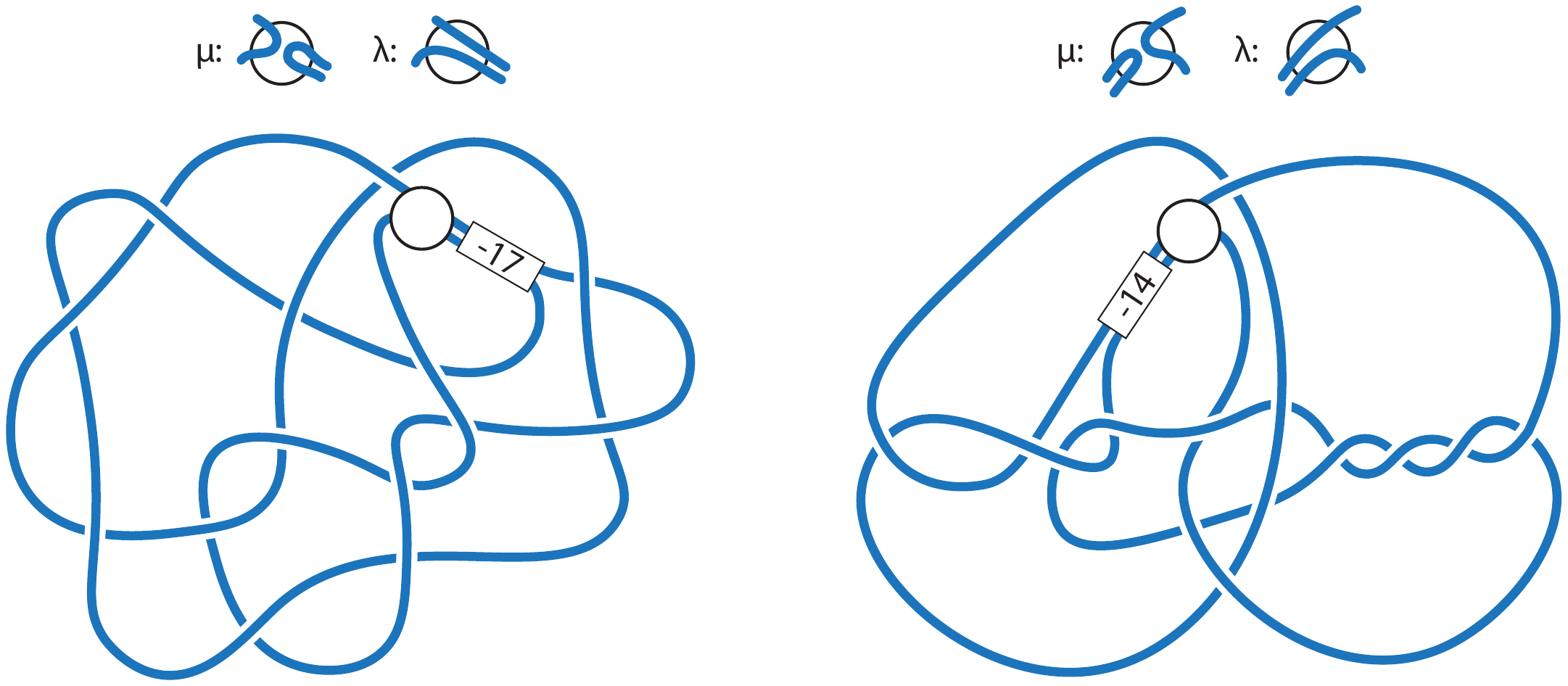}
	\caption{The tangle exteriors $T_1$ of $K_1$ (left) and $T_2$ of $K_2$ (right), framed by the images of the meridians $\mu$ and Seifert longitudes $\lambda$ under the covering map.  Above each are the associated rational tangle fillings.}
	\label{fig:tangle_exteriors}
\end{figure}
								
\begin{prop}\label{prop:str_inv}
    The $\varkappa$-invariant of $K$ is given as
\[
\begin{tikzpicture}[scale=0.85]
\draw [black] (-2.2,0) -- (-2.2,1.5);
\draw [black] (-1,0) -- (-1,1.5);
\draw [gray] (-0.5,0) -- (-0.5,1.5);
\draw [gray] (0,0) -- (0,1.5);
\draw [gray] (0.5,0) -- (0.5,1.5);
\draw [gray] (1,0) -- (1,1.5);
\draw [gray] (1.5,0) -- (1.5,1.5);
\draw [gray] (2,0) -- (2,1.5);
\draw [gray] (2.5,0) -- (2.5,1.5);
\draw [gray] (3,0) -- (3,1.5);
\draw [gray] (3.5,0) -- (3.5,1.5);
\draw [black] (4,0) -- (4,1.5);

\draw [black] (-2.2,1.5) -- (4,1.5);
\draw [black] (-2.2,1) -- (4,1);
\draw [gray] (-2.2,0.5) -- (4,0.5);
\draw [black] (-2.2,0) -- (4,0);

\node at (-3.2,0.75) {$\varkappa(K_1)=$};

\node at (-1.58,0.75) {\footnotesize{$\delta=17$}}; 
\node at (-1.58,0.25) {\footnotesize{$\delta=15$}}; 

\node at (-1.83,1.25) {\footnotesize{$h=$}};

\node at (-0.75,1.25) {\footnotesize{-$9$}}; 
\node at (-0.25,1.25) {\footnotesize{-$8$}}; 
\node at (0.25,1.25) {\footnotesize{-$7$}}; 
\node at (0.75,1.25) {\footnotesize{-$6$}}; 
\node at (1.25,1.25) {\footnotesize{-$5$}}; 
\node at (1.75,1.25) {\footnotesize{-$4$}}; 
\node at (2.25,1.25) {\footnotesize{-$3$}}; 
\node at (2.75,1.25) {\footnotesize{-$2$}}; 
\node at (3.25,1.25) {\footnotesize{-$1$}}; 
\node at (3.75,1.25) {\footnotesize{$0$}}; 

\node at (-0.75,.75) {$1$}; 
\node at (-0.25,.75) {$2$}; 
\node at (0.25,.75) {$3$}; 
\node at (0.75,.75) {$4$}; 
\node at (1.25,.75) {$4$}; 
\node at (1.75,.75) {$4$}; 
\node at (2.25,.75) {$3$}; 
\node at (2.75,.75) {$2$}; 
\node at (3.25,.75) {$1$}; 

\node at (1.25,.25) {$1$}; 
\node at (2.25,.25) {$1$}; 
\node at (2.75,.25) {$1$}; 
\node at (3.75,.25) {$1$}; 
\end{tikzpicture}\]
\[
\begin{tikzpicture}[scale=0.85]
\draw [black] (-2.2,0) -- (-2.2,1.5);
\draw [black] (-1,0) -- (-1,1.5);
\draw [gray] (-0.5,0) -- (-0.5,1.5);
\draw [gray] (0,0) -- (0,1.5);
\draw [gray] (0.5,0) -- (0.5,1.5);
\draw [gray] (1,0) -- (1,1.5);
\draw [gray] (1.5,0) -- (1.5,1.5);
\draw [gray] (2,0) -- (2,1.5);
\draw [gray] (2.5,0) -- (2.5,1.5);
\draw [gray] (3,0) -- (3,1.5);
\draw [gray] (3.5,0) -- (3.5,1.5);
\draw [gray] (4,0) -- (4,1.5);
\draw [black] (4.5,0) -- (4.5,1.5);

\draw [black] (-2.2,1.5) -- (4.5,1.5);
\draw [black] (-2.2,1) -- (4.5,1);
\draw [gray] (-2.2,0.5) -- (4.5,0.5);
\draw [black] (-2.2,0) -- (4.5,0);

\node at (-3.2,0.75) {$\varkappa(K_2)=$};

\node at (-1.58,0.75) {\footnotesize{$\delta=17$}}; 
\node at (-1.58,0.25) {\footnotesize{$\delta=15$}}; 

\node at (-1.83,1.25) {\footnotesize{$h=$}};

\node at (-0.75,1.25) {\footnotesize{-$6$}}; 
\node at (-0.25,1.25) {\footnotesize{-$5$}}; 
\node at (0.25,1.25) {\footnotesize{-$4$}}; 
\node at (0.75,1.25) {\footnotesize{-$3$}}; 
\node at (1.25,1.25) {\footnotesize{-$2$}}; 
\node at (1.75,1.25) {\footnotesize{-$1$}}; 
\node at (2.25,1.25) {\footnotesize{$0$}}; 
\node at (2.75,1.25) {\footnotesize{$1$}}; 
\node at (3.25,1.25) {\footnotesize{$2$}}; 
\node at (3.75,1.25) {\footnotesize{$3$}}; 
\node at (4.25,1.25) {\footnotesize{$4$}}; 

\node at (-0.75,.75) {$1$}; 
\node at (-0.25,.75) {$2$}; 
\node at (0.25,.75) {$3$}; 
\node at (0.75,.75) {$4$}; 
\node at (1.25,.75) {$4$}; 
\node at (1.75,.75) {$4$}; 
\node at (2.25,.75) {$3$}; 
\node at (2.75,.75) {$2$}; 
\node at (3.25,.75) {$1$}; 

\node at (1.25,.25) {$1$};
\node at (1.75,.25) {$1$}; 
\node at (2.25,.25) {$1$}; 
\node at (2.75,.25) {$2$}; 
\node at (3.25,.25) {$1$}; 
\node at (3.75,.25) {$1$}; 
\node at (4.25,.25) {$1$}; 
\end{tikzpicture}\]
Moreover, for any $r\in \Q$, the branching set $T(r)$ has Khovanov width $2$ and is therefore not thin. 
\end{prop}
				
\begin{proof} 
With the Mathematica KnotTheory package~\cite{knotatlas} and KnotJob~\cite{KnotJob} we compute the reduced Khovanov homology of $T(n)$ for small integers $n$ until we find the integer $N$. By comparing the Khovanov homologies of $T(N-1)$ and $T(N+1)$ we can read-off the $\varkappa$-invariant.

These computations are shown in Figure~\ref{fig:Kappa}. The Khovanov homologies of $T(n)$ are displayed by printing the generators (as $\bullet$ or $\circ$) in the different homological gradings. There are two relative $\delta$-gradings, distinguished here by the conventions that a $\bullet$ denotes a copy of $\Z_2$ in relative $\delta$-grading $17$ while a $\circ$ denotes a copy of $\Z_2$ in relative $\delta$-grading $15$. 

From this, we can read-off that $N=20$ for $K_1$ and $N=16$ for $K_2$. Thus the generators in black (and without an underline) appear in $Kh(T(n))$ for all integers $n$ and therefore contribute to $\varkappa$, while the generators in red (with an underline) do not survive to $\varkappa$. This yields the claimed values for $\varkappa$.
				
The Khovanov width of $K(r)$ can be computed using the results of \cite{Watson2012}. The computations in Figure~\ref{fig:Kappa} combined with \cite[Lemma~4.10]{Watson2012} imply that each link $T(n)$ has Khovanov width ${w}_{Kh}=2$.
By \cite[Proposition~5.2]{Watson2012}, this implies that the Khovanov width of all rational tangle fillings $T(r)$ is $2$ as well. 
\end{proof}

\begin{figure}[tt]
 \caption{The calculation of the $\varkappa$-invariant of $K_1$ and $K_2$.}
\begin{tikzpicture}[>=latex] 

\draw [black] (5.5,0) -- (5.5,8);
\draw [gray] (4.5,0) -- (4.5,8);
\draw [gray] (3.5,0) -- (3.5,8);
\draw [gray] (2.5,0) -- (2.5,8);
\draw [gray] (1.5,0) -- (1.5,8);
\draw [gray] (0.5,0) -- (0.5,8);
\draw [gray] (-0.5,0) -- (-0.5,8);
\draw [gray] (-1.5,0) -- (-1.5,8);
\draw [gray] (-2.5,0) -- (-2.5,8);
\draw [gray] (-3.5,0) -- (-3.5,8);
\draw [black] (-4.5,0) -- (-4.5,8);
\draw [black] (-6,0) -- (-6,8);

\draw [black] (-6,8) -- (5.5,8);
\draw [black] (-6,7) -- (5.5,7);
\draw [gray] (-6,6.5) -- (5.5,6.5);
\draw [gray] (-6,5.5) -- (5.5,5.5);
\draw [gray] (-6,4.5) -- (5.5,4.5);
\draw [gray] (-6,3.5) -- (5.5,3.5);
\draw [gray] (-6,2.5) -- (5.5,2.5);
\draw [gray] (-6,1.5) -- (5.5,1.5);
\draw [black] (-6,1) -- (5.5,1);
\draw [black] (-6,0) -- (5.5,0);

\draw [black] (-6,8) -- (-4.5,7);

\node at (-5.5,7.25) {\footnotesize{$n$}};
\node at (-5,7.75) {\footnotesize{$h$}};
\node at (-4,7.5) {\footnotesize{-$9$}};
\node at (-3,7.5) {\footnotesize{-$8$}};
\node at (-2,7.5) {\footnotesize{-$7$}};
\node at (-1,7.5) {\footnotesize{-$6$}};
\node at (0,7.5) {\footnotesize{-$5$}};
\node at (1,7.5) {\footnotesize{-$4$}};
\node at (2,7.5) {\footnotesize{-$3$}};
\node at (3,7.5) {\footnotesize{-$2$}};
\node at (4,7.5) {\footnotesize{-$1$}};
\node at (5,7.5) {\footnotesize{$0$}};

\node at (-5.25, 6) {\footnotesize{$Kh(T(22))$}};

\node at (-4.2, 6.15) {\color{black}{$\bullet$}};

\node at (-3.2, 6.15) {\color{black}{$\bullet$}};
\node at (-3, 6.15) {\color{black}{$\bullet$}};

\node at (-2.2, 6.15) {\color{black}{$\bullet$}};
\node at (-2, 6.15) {\color{black}{$\bullet$}};
\node at (-1.8, 6.15) {\color{black}{$\bullet$}};

\node at (-1.2, 6.15) {\color{black}{$\bullet$}};
\node at (-1, 6.15) {\color{black}{$\bullet$}};
\node at (-0.8, 6.15) {\color{black}{$\bullet$}};
\node at (-1.2, 5.85) {\color{black}{$\bullet$}};
\node at (-1, 5.85) {\color{red}{$\bullet$}};
  \node at (-1, 5.85) {\color{red}{\underline{\phantom{$\bullet$}}}};

\node at (-0.2, 6.15) {\color{black}{$\bullet$}};
\node at (0, 6.15) {\color{black}{$\bullet$}};
\node at (0.2, 6.15) {\color{black}{$\bullet$}};
\node at (-0.2, 5.85) {\color{black}{$\bullet$}};
\node at (0, 5.85) {\color{black}{$\circ$}};

\node at (0.8, 6.15) {\color{black}{$\bullet$}};
\node at (1, 6.15) {\color{black}{$\bullet$}};
\node at (1.2, 6.15) {\color{black}{$\bullet$}};
\node at (0.8, 5.85) {\color{black}{$\bullet$}};
\node at (1, 5.85) {\color{red}{$\bullet$}};
\node at (1, 5.85) {\color{red}{\underline{\phantom{$\bullet$}}}};

\node at (1.8, 6.15) {\color{black}{$\bullet$}};
\node at (2, 6.15) {\color{black}{$\bullet$}};
\node at (2.2, 6.15) {\color{black}{$\bullet$}};
\node at (1.8, 5.85) {\color{black}{$\circ$}};

\node at (2.8, 6.15) {\color{black}{$\bullet$}};
\node at (3, 6.15) {\color{black}{$\bullet$}};
\node at (3.2, 6.15) {\color{black}{$\circ$}};

\node at (3.8, 6.15) {\color{black}{$\bullet$}};

\node at (4.8, 6.15) {\color{black}{$\circ$}};

\node at (-5.25, 5) {\footnotesize{$Kh(T(21))$}};

\node at (-4.2, 5.15) {\color{black}{$\bullet$}};

\node at (-3.2, 5.15) {\color{black}{$\bullet$}};
\node at (-3, 5.15) {\color{black}{$\bullet$}};

\node at (-2.2, 5.15) {\color{black}{$\bullet$}};
\node at (-2, 5.15) {\color{black}{$\bullet$}};
\node at (-1.8, 5.15) {\color{black}{$\bullet$}};

\node at (-1.2, 5.15) {\color{black}{$\bullet$}};
\node at (-1, 5.15) {\color{black}{$\bullet$}};
\node at (-0.8, 5.15) {\color{black}{$\bullet$}};
\node at (-1.2, 4.85) {\color{black}{$\bullet$}};

\node at (-0.2, 5.15) {\color{black}{$\bullet$}};
\node at (0, 5.15) {\color{black}{$\bullet$}};
\node at (0.2, 5.15) {\color{black}{$\bullet$}};
\node at (-0.2, 4.85) {\color{black}{$\bullet$}};
\node at (0, 4.85) {\color{black}{$\circ$}};

\node at (0.8, 5.15) {\color{black}{$\bullet$}};
\node at (1, 5.15) {\color{black}{$\bullet$}};
\node at (1.2, 5.15) {\color{black}{$\bullet$}};
\node at (0.8, 4.85) {\color{black}{$\bullet$}};
\node at (1, 4.85) {\color{red}{$\bullet$}};
\node at (1, 4.85) {\color{red}{\underline{\phantom{$\bullet$}}}};

\node at (1.8, 5.15) {\color{black}{$\bullet$}};
\node at (2, 5.15) {\color{black}{$\bullet$}};
\node at (2.2, 5.15) {\color{black}{$\bullet$}};
 \node at (1.8, 4.85) {\color{black}{$\circ$}};

\node at (2.8, 5.15) {\color{black}{$\bullet$}};
\node at (3, 5.15) {\color{black}{$\bullet$}};
\node at (3.2, 5.15) {\color{black}{$\circ$}};

\node at (3.8, 5.15) {\color{black}{$\bullet$}};

\node at (4.8, 5.15) {\color{black}{$\circ$}};

\node at (-5.25, 4) {\footnotesize{$Kh(T(20))$}};

\node at (-4.2, 4.15) {\color{black}{$\bullet$}};

\node at (-3.2, 4.15) {\color{black}{$\bullet$}};
\node at (-3, 4.15) {\color{black}{$\bullet$}};

\node at (-2.2, 4.15) {\color{black}{$\bullet$}};
\node at (-2, 4.15) {\color{black}{$\bullet$}};
\node at (-1.8, 4.15) {\color{black}{$\bullet$}};

\node at (-1.2, 4.15) {\color{black}{$\bullet$}};
\node at (-1, 4.15) {\color{black}{$\bullet$}};
\node at (-0.8, 4.15) {\color{black}{$\bullet$}};
\node at (-1.2, 3.85) {\color{black}{$\bullet$}};

\node at (-0.2, 4.15) {\color{black}{$\bullet$}};
\node at (0, 4.15) {\color{black}{$\bullet$}};
\node at (0.2, 4.15) {\color{black}{$\bullet$}};
\node at (-0.2, 3.85) {\color{black}{$\bullet$}};
\node at (0, 3.85) {\color{black}{$\circ$}};

\node at (0.8, 4.15) {\color{black}{$\bullet$}};
\node at (1, 4.15) {\color{black}{$\bullet$}};
\node at (1.2, 4.15) {\color{black}{$\bullet$}};
\node at (0.8, 3.85) {\color{black}{$\bullet$}};
\node at (1, 3.85) {\color{red}{$\bullet$}};
\node at (1, 3.85) {\color{red}{\underline{\phantom{$\bullet$}}}};
\node at (1.2, 3.85) {\color{red}{$\circ$}};
\node at (1.2, 3.85) {\color{red}{\underline{\phantom{$\circ$}}}};

\node at (1.8, 4.15) {\color{black}{$\bullet$}};
\node at (2, 4.15) {\color{black}{$\bullet$}};
\node at (2.2, 4.15) {\color{black}{$\bullet$}};
\node at (1.8, 3.85) {\color{black}{$\circ$}};

\node at (2.8, 4.15) {\color{black}{$\bullet$}};
\node at (3, 4.15) {\color{black}{$\bullet$}};
\node at (3.2, 4.15) {\color{black}{$\circ$}};

\node at (3.8, 4.15) {\color{black}{$\bullet$}};

\node at (4.8, 4.15) {\color{black}{$\circ$}};

\node at (-5.25, 3) {\footnotesize{$Kh(T(19))$}};

\node at (-4.2, 3.15) {\color{black}{$\bullet$}};

\node at (-3.2, 3.15) {\color{black}{$\bullet$}};
\node at (-3, 3.15) {\color{black}{$\bullet$}};

\node at (-2.2, 3.15) {\color{black}{$\bullet$}};
\node at (-2, 3.15) {\color{black}{$\bullet$}};
\node at (-1.8, 3.15) {\color{black}{$\bullet$}};

\node at (-1.2, 3.15) {\color{black}{$\bullet$}};
\node at (-1, 3.15) {\color{black}{$\bullet$}};
\node at (-0.8, 3.15) {\color{black}{$\bullet$}};
\node at (-1.2, 2.85) {\color{black}{$\bullet$}};

\node at (-0.2, 3.15) {\color{black}{$\bullet$}};
\node at (0, 3.15) {\color{black}{$\bullet$}};
\node at (0.2, 3.15) {\color{black}{$\bullet$}};
\node at (-0.2, 2.85) {\color{black}{$\bullet$}};
\node at (0, 2.85) {\color{black}{$\circ$}};

\node at (0.8, 3.15) {\color{black}{$\bullet$}};
\node at (1, 3.15) {\color{black}{$\bullet$}};
\node at (1.2, 3.15) {\color{black}{$\bullet$}};
\node at (0.8, 2.85) {\color{black}{$\bullet$}};
 \node at (1, 2.85) {\color{red}{$\circ$}};
\node at (1, 2.85) {\color{red}{\underline{\phantom{$\circ$}}}};

\node at (1.8, 3.15) {\color{black}{$\bullet$}};
\node at (2, 3.15) {\color{black}{$\bullet$}};
\node at (2.2, 3.15) {\color{black}{$\bullet$}};
\node at (1.8, 2.85) {\color{black}{$\circ$}};

\node at (2.8, 3.15) {\color{black}{$\bullet$}};
\node at (3, 3.15) {\color{black}{$\bullet$}};
\node at (3.2, 3.15) {\color{black}{$\circ$}};

\node at (3.8, 3.15) {\color{black}{$\bullet$}};

\node at (4.8, 3.15) {\color{black}{$\circ$}};

\node at (-5.25, 2) {\footnotesize{$Kh(T(18))$}};

\node at (-4.2, 2.15) {\color{black}{$\bullet$}};

\node at (-3.2, 2.15) {\color{black}{$\bullet$}};
\node at (-3, 2.15) {\color{black}{$\bullet$}};

\node at (-2.2, 2.15) {\color{black}{$\bullet$}};
\node at (-2, 2.15) {\color{black}{$\bullet$}};
\node at (-1.8, 2.15) {\color{black}{$\bullet$}};

\node at (-1.2, 2.15) {\color{black}{$\bullet$}};
\node at (-1, 2.15) {\color{black}{$\bullet$}};
\node at (-0.8, 2.15) {\color{black}{$\bullet$}};
\node at (-1.2, 1.85) {\color{black}{$\bullet$}};

\node at (-0.2, 2.15) {\color{black}{$\bullet$}};
\node at (0, 2.15) {\color{black}{$\bullet$}};
\node at (0.2, 2.15) {\color{black}{$\bullet$}};
\node at (-0.2, 1.85) {\color{black}{$\bullet$}};
\node at (0, 1.85) {\color{black}{$\circ$}};

\node at (0.8, 2.15) {\color{black}{$\bullet$}};
\node at (1, 2.15) {\color{black}{$\bullet$}};
\node at (1.2, 2.15) {\color{black}{$\bullet$}};
\node at (0.8, 1.85) {\color{black}{$\bullet$}};
\node at (1, 1.85) {\color{red}{$\circ$}};
\node at (1, 1.85) {\color{red}{\underline{\phantom{$\circ$}}}};

\node at (1.8, 2.15) {\color{black}{$\bullet$}};
\node at (2, 2.15) {\color{black}{$\bullet$}};
\node at (2.2, 2.15) {\color{black}{$\bullet$}};
\node at (1.8, 1.85) {\color{black}{$\circ$}};

\node at (2.8, 2.15) {\color{black}{$\bullet$}};
\node at (3, 2.15) {\color{black}{$\bullet$}};
\node at (3.2, 2.15) {\color{black}{$\circ$}};
\node at (2.8, 1.85) {\color{red}{$\circ$}};
\node at (2.8, 1.85) {\color{red}{\underline{\phantom{$\circ$}}}};

\node at (3.8, 2.15) {\color{black}{$\bullet$}};

\node at (4.8, 2.15) {\color{black}{$\circ$}};

\node at (-5.25, 0.5) {{$\varkappa(K_1)$}};

\node at (-4.2, 0.65) {\color{black}{$\bullet$}};

\node at (-3.2, 0.65) {\color{black}{$\bullet$}};
\node at (-3, 0.65) {\color{black}{$\bullet$}};

\node at (-2.2, 0.65) {\color{black}{$\bullet$}};
\node at (-2, 0.65) {\color{black}{$\bullet$}};
\node at (-1.8, 0.65) {\color{black}{$\bullet$}};

\node at (-1.2, 0.65) {\color{black}{$\bullet$}};
\node at (-1, 0.65) {\color{black}{$\bullet$}};
\node at (-0.8, 0.65) {\color{black}{$\bullet$}};
\node at (-1.2, 0.35) {\color{black}{$\bullet$}};

\node at (-0.2, 0.65) {\color{black}{$\bullet$}};
\node at (0, 0.65) {\color{black}{$\bullet$}};
\node at (0.2, 0.65) {\color{black}{$\bullet$}};
\node at (-0.2,0.35) {\color{black}{$\bullet$}};
\node at (0, 0.35) {\color{black}{$\circ$}};

\node at (0.8, 0.65) {\color{black}{$\bullet$}};
\node at (1, 0.65) {\color{black}{$\bullet$}};
\node at (1.2, 0.655) {\color{black}{$\bullet$}};
\node at (0.8, 0.35) {\color{black}{$\bullet$}};

\node at (1.8, 0.65) {\color{black}{$\bullet$}};
\node at (2, 0.65) {\color{black}{$\bullet$}};
\node at (2.2, 0.65) {\color{black}{$\bullet$}};
\node at (1.8, 0.35) {\color{black}{$\circ$}};

\node at (2.8, 0.65) {\color{black}{$\bullet$}};
\node at (3, 0.65) {\color{black}{$\bullet$}};
\node at (3.2, 0.65) {\color{black}{$\circ$}};

\node at (3.8, 0.65) {\color{black}{$\bullet$}};

\node at (4.8, 0.65) {\color{black}{$\circ$}};

\draw[thin,->>] (-5.15,6.75) -- (-5.15,6.25);
\draw[thin,->>] (-5.15,5.75) -- (-5.15,5.25);
\draw[thin,left hook->] (-5.15,4.75) -- (-5.15,4.25);
\draw[thin,->>] (-5.15,3.75) -- (-5.15,3.25);
\draw[thin,left hook->] (-5.15,2.75) -- (-5.15,2.25);
\draw[thin,left hook->] (-5.15,1.75) -- (-5.15,1.25);
\end{tikzpicture}

\hfill

\begin{tikzpicture}[>=latex] 

\draw [black] (6.5,0) -- (6.5,8);
\draw [gray] (5.5,0) -- (5.5,8);
\draw [gray] (4.5,0) -- (4.5,8);
\draw [gray] (3.5,0) -- (3.5,8);
\draw [gray] (2.5,0) -- (2.5,8);
\draw [gray] (1.5,0) -- (1.5,8);
\draw [gray] (0.5,0) -- (0.5,8);
\draw [gray] (-0.5,0) -- (-0.5,8);
\draw [gray] (-1.5,0) -- (-1.5,8);
\draw [gray] (-2.5,0) -- (-2.5,8);
\draw [gray] (-3.5,0) -- (-3.5,8);
\draw [black] (-4.5,0) -- (-4.5,8);
\draw [black] (-6,0) -- (-6,8);

\draw [black] (-6,8) -- (6.5,8);
\draw [black] (-6,7) -- (6.5,7);
\draw [gray] (-6,6.5) -- (6.5,6.5);
\draw [gray] (-6,5.5) -- (6.5,5.5);
\draw [gray] (-6,4.5) -- (6.5,4.5);
\draw [gray] (-6,3.5) -- (6.5,3.5);
\draw [gray] (-6,2.5) -- (6.5,2.5);
\draw [gray] (-6,1.5) -- (6.5,1.5);
\draw [black] (-6,1) -- (6.5,1);
\draw [black] (-6,0) -- (6.5,0);

\draw [black] (-6,8) -- (-4.5,7);

\node at (-5.5,7.25) {\footnotesize{$n$}};
\node at (-5,7.75) {\footnotesize{$h$}};
\node at (-4,7.5) {\footnotesize{-$6$}};
\node at (-3,7.5) {\footnotesize{-$5$}};
\node at (-2,7.5) {\footnotesize{-$4$}};
\node at (-1,7.5) {\footnotesize{-$3$}};
\node at (0,7.5) {\footnotesize{-$2$}};
\node at (1,7.5) {\footnotesize{-$1$}};
\node at (2,7.5) {\footnotesize{$0$}};
\node at (3,7.5) {\footnotesize{$1$}};
\node at (4,7.5) {\footnotesize{$2$}};
\node at (5,7.5) {\footnotesize{$3$}};
\node at (6,7.5) {\footnotesize{$4$}};

\node at (-5.25, 6) {\footnotesize{$Kh(T(18))$}};

\node at (-4.2, 6.15) {\color{black}{$\bullet$}};

\node at (-3.2, 6.15) {\color{black}{$\bullet$}};
\node at (-3, 6.15) {\color{black}{$\bullet$}};

\node at (-2.2, 6.15) {\color{black}{$\bullet$}};
\node at (-2, 6.15) {\color{black}{$\bullet$}};
\node at (-1.8, 6.15) {\color{black}{$\bullet$}};

\node at (-1.2, 6.15) {\color{black}{$\bullet$}};
\node at (-1, 6.15) {\color{black}{$\bullet$}};
\node at (-0.8, 6.15) {\color{black}{$\bullet$}};
\node at (-1.2, 5.85) {\color{black}{$\bullet$}};

\node at (-0.2, 6.15) {\color{black}{$\bullet$}};
\node at (0, 6.15) {\color{black}{$\bullet$}};
\node at (0.2, 6.15) {\color{black}{$\bullet$}};
\node at (-0.2, 5.85) {\color{black}{$\bullet$}};
\node at (0, 5.85) {\color{black}{$\circ$}};
\node at (0.2, 5.85) {\color{red}{$\bullet$}};
\node at (0.2, 5.85) {\color{red}{\underline{\phantom{$\bullet$}}}};

\node at (0.8, 6.15) {\color{black}{$\bullet$}};
\node at (1, 6.15) {\color{black}{$\bullet$}};
\node at (1.2, 6.15) {\color{black}{$\bullet$}};
\node at (0.8, 5.85) {\color{black}{$\bullet$}};
\node at (1, 5.85) {\color{black}{$\circ$}};

\node at (1.8, 6.15) {\color{black}{$\bullet$}};
\node at (2, 6.15) {\color{black}{$\bullet$}};
\node at (2.2, 6.15) {\color{black}{$\bullet$}};
\node at (1.8, 5.85) {\color{black}{$\circ$}};
\node at (2, 5.85) {\color{red}{$\bullet$}};
\node at (2, 5.85) {\color{red}{\underline{\phantom{$\bullet$}}}};

\node at (2.8, 6.15) {\color{black}{$\bullet$}};
\node at (3, 6.15) {\color{black}{$\bullet$}};
\node at (3.2, 6.15) {\color{black}{$\circ$}};
\node at (2.8, 5.85) {\color{black}{$\circ$}};

\node at (3.8, 6.15) {\color{black}{$\bullet$}};
\node at (4, 6.15) {\color{black}{$\circ$}};

\node at (4.8, 6.15) {\color{black}{$\circ$}};

\node at (5.8, 6.15) {\color{black}{$\circ$}};

\node at (-5.25, 5) {\footnotesize{$Kh(T(17))$}};

\node at (-4.2, 5.15) {\color{black}{$\bullet$}};

\node at (-3.2, 5.15) {\color{black}{$\bullet$}};
\node at (-3, 5.15) {\color{black}{$\bullet$}};

\node at (-2.2, 5.15) {\color{black}{$\bullet$}};
\node at (-2, 5.15) {\color{black}{$\bullet$}};
\node at (-1.8, 5.15) {\color{black}{$\bullet$}};

\node at (-1.2, 5.15) {\color{black}{$\bullet$}};
\node at (-1, 5.15) {\color{black}{$\bullet$}};
\node at (-0.8, 5.15) {\color{black}{$\bullet$}};
\node at (-1.2, 4.85) {\color{black}{$\bullet$}};

\node at (-0.2, 5.15) {\color{black}{$\bullet$}};
\node at (0, 5.15) {\color{black}{$\bullet$}};
\node at (0.2, 5.15) {\color{black}{$\bullet$}};
\node at (-0.2, 4.85) {\color{black}{$\bullet$}};
\node at (0, 4.85) {\color{black}{$\circ$}};

\node at (0.8, 5.15) {\color{black}{$\bullet$}};
\node at (1, 5.15) {\color{black}{$\bullet$}};
\node at (1.2, 5.15) {\color{black}{$\bullet$}};
\node at (0.8, 4.85) {\color{black}{$\bullet$}};
\node at (1, 4.85) {\color{black}{$\circ$}};

\node at (1.8, 5.15) {\color{black}{$\bullet$}};
\node at (2, 5.15) {\color{black}{$\bullet$}};
\node at (2.2, 5.15) {\color{black}{$\bullet$}};
\node at (1.8, 4.85) {\color{black}{$\circ$}};
\node at (2, 4.85) {\color{red}{$\bullet$}};
\node at (2, 4.85) {\color{red}{\underline{\phantom{$\bullet$}}}};

\node at (2.8, 5.15) {\color{black}{$\bullet$}};
\node at (3, 5.15) {\color{black}{$\bullet$}};
\node at (3.2, 5.15) {\color{black}{$\circ$}};
\node at (2.8, 4.85) {\color{black}{$\circ$}};

\node at (3.8, 5.15) {\color{black}{$\bullet$}};
\node at (4, 5.15) {\color{black}{$\circ$}};

\node at (4.8, 5.15) {\color{black}{$\circ$}};

\node at (5.8, 5.15) {\color{black}{$\circ$}};

\node at (-5.25, 4) {\footnotesize{$Kh(T(16))$}};

\node at (-4.2, 4.15) {\color{black}{$\bullet$}};

\node at (-3.2, 4.15) {\color{black}{$\bullet$}};
\node at (-3, 4.15) {\color{black}{$\bullet$}};

\node at (-2.2, 4.15) {\color{black}{$\bullet$}};
\node at (-2, 4.15) {\color{black}{$\bullet$}};
\node at (-1.8, 4.15) {\color{black}{$\bullet$}};

\node at (-1.2, 4.15) {\color{black}{$\bullet$}};
\node at (-1, 4.15) {\color{black}{$\bullet$}};
\node at (-0.8, 4.15) {\color{black}{$\bullet$}};
\node at (-1.2, 3.85) {\color{black}{$\bullet$}};

\node at (-0.2, 4.15) {\color{black}{$\bullet$}};
\node at (0, 4.15) {\color{black}{$\bullet$}};
\node at (0.2, 4.15) {\color{black}{$\bullet$}};
\node at (-0.2, 3.85) {\color{black}{$\bullet$}};
\node at (0, 3.85) {\color{black}{$\circ$}};

\node at (0.8, 4.15) {\color{black}{$\bullet$}};
\node at (1, 4.15) {\color{black}{$\bullet$}};
\node at (1.2, 4.15) {\color{black}{$\bullet$}};
\node at (0.8, 3.85) {\color{black}{$\bullet$}};
\node at (1, 3.85) {\color{black}{$\circ$}};

\node at (1.8, 4.15) {\color{black}{$\bullet$}};
\node at (2, 4.15) {\color{black}{$\bullet$}};
\node at (2.2, 4.15) {\color{black}{$\bullet$}};
\node at (1.8, 3.85) {\color{black}{$\circ$}};
\node at (2, 3.85) {\color{red}{$\bullet$}};
\node at (2, 3.85) {\color{red}{\underline{\phantom{$\bullet$}}}};
\node at (2.2, 3.85) {\color{red}{$\circ$}};
\node at (2.2, 3.85) {\color{red}{\underline{\phantom{$\circ$}}}};

\node at (2.8, 4.15) {\color{black}{$\bullet$}};
\node at (3, 4.15) {\color{black}{$\bullet$}};
\node at (3.2, 4.15) {\color{black}{$\circ$}};
\node at (2.8, 3.85) {\color{black}{$\circ$}};

\node at (3.8, 4.15) {\color{black}{$\bullet$}};
\node at (4, 4.15) {\color{black}{$\circ$}};

\node at (4.8, 4.15) {\color{black}{$\circ$}};

\node at (5.8, 4.15) {\color{black}{$\circ$}};

\node at (-5.25, 3) {\footnotesize{$Kh(T(15))$}};

\node at (-4.2, 3.15) {\color{black}{$\bullet$}};

\node at (-3.2, 3.15) {\color{black}{$\bullet$}};
\node at (-3, 3.15) {\color{black}{$\bullet$}};

\node at (-2.2, 3.15) {\color{black}{$\bullet$}};
\node at (-2, 3.15) {\color{black}{$\bullet$}};
\node at (-1.8, 3.15) {\color{black}{$\bullet$}};

\node at (-1.2, 3.15) {\color{black}{$\bullet$}};
\node at (-1, 3.15) {\color{black}{$\bullet$}};
\node at (-0.8, 3.15) {\color{black}{$\bullet$}};
\node at (-1.2, 2.85) {\color{black}{$\bullet$}};

\node at (-0.2, 3.15) {\color{black}{$\bullet$}};
\node at (0, 3.15) {\color{black}{$\bullet$}};
\node at (0.2, 3.15) {\color{black}{$\bullet$}};
\node at (-0.2, 2.85) {\color{black}{$\bullet$}};
\node at (0, 2.85) {\color{black}{$\circ$}};

\node at (0.8, 3.15) {\color{black}{$\bullet$}};
\node at (1, 3.15) {\color{black}{$\bullet$}};
\node at (1.2, 3.15) {\color{black}{$\bullet$}};
\node at (0.8, 2.85) {\color{black}{$\bullet$}};
\node at (1, 2.85) {\color{black}{$\circ$}};

\node at (1.8, 3.15) {\color{black}{$\bullet$}};
\node at (2, 3.15) {\color{black}{$\bullet$}};
\node at (2.2, 3.15) {\color{black}{$\bullet$}};
\node at (1.8, 2.85) {\color{black}{$\circ$}};
\node at (2, 2.85) {\color{red}{$\circ$}};
\node at (2, 2.85) {\color{red}{\underline{\phantom{$\circ$}}}};

\node at (2.8, 3.15) {\color{black}{$\bullet$}};
\node at (3, 3.15) {\color{black}{$\bullet$}};
\node at (3.2, 3.15) {\color{black}{$\circ$}};
\node at (2.8, 2.85) {\color{black}{$\circ$}};

\node at (3.8, 3.15) {\color{black}{$\bullet$}};
\node at (4, 3.15) {\color{black}{$\circ$}};

\node at (4.8, 3.15) {\color{black}{$\circ$}};

\node at (5.8, 3.15) {\color{black}{$\circ$}};

\node at (-5.25, 2) {\footnotesize{$Kh(T(14))$}};

\node at (-4.2, 2.15) {\color{black}{$\bullet$}};

\node at (-3.2, 2.15) {\color{black}{$\bullet$}};
\node at (-3, 2.15) {\color{black}{$\bullet$}};

\node at (-2.2, 2.15) {\color{black}{$\bullet$}};
\node at (-2, 2.15) {\color{black}{$\bullet$}};
\node at (-1.8, 2.15) {\color{black}{$\bullet$}};

\node at (-1.2, 2.15) {\color{black}{$\bullet$}};
\node at (-1, 2.15) {\color{black}{$\bullet$}};
\node at (-0.8, 2.15) {\color{black}{$\bullet$}};
\node at (-1.2, 1.85) {\color{black}{$\bullet$}};

\node at (-0.2, 2.15) {\color{black}{$\bullet$}};
\node at (0, 2.15) {\color{black}{$\bullet$}};
\node at (0.2, 2.15) {\color{black}{$\bullet$}};
\node at (-0.2, 1.85) {\color{black}{$\bullet$}};
\node at (0, 1.85) {\color{black}{$\circ$}};

\node at (0.8, 2.15) {\color{black}{$\bullet$}};
\node at (1, 2.15) {\color{black}{$\bullet$}};
\node at (1.2, 2.15) {\color{black}{$\bullet$}};
\node at (0.8, 1.85) {\color{black}{$\bullet$}};
\node at (1, 1.85) {\color{black}{$\circ$}};

\node at (1.8, 2.15) {\color{black}{$\bullet$}};
\node at (2, 2.15) {\color{black}{$\bullet$}};
\node at (2.2, 2.15) {\color{black}{$\bullet$}};
\node at (1.8, 1.85) {\color{black}{$\circ$}};
\node at (2, 1.85) {\color{red}{$\circ$}};
\node at (2, 1.85) {\color{red}{\underline{\phantom{$\circ$}}}};

\node at (2.8, 2.15) {\color{black}{$\bullet$}};
\node at (3, 2.15) {\color{black}{$\bullet$}};
\node at (3.2, 2.15) {\color{black}{$\circ$}};
\node at (2.8, 1.85) {\color{black}{$\circ$}};

\node at (3.8, 2.15) {\color{black}{$\bullet$}};
\node at (4, 2.15) {\color{black}{$\circ$}};
\node at (4.2, 2.15) {\color{red}{$\circ$}};
\node at (4.2, 2.15) {\color{red}{\underline{\phantom{$\circ$}}}};

\node at (4.8, 2.15) {\color{black}{$\circ$}};

\node at (5.8, 2.15) {\color{black}{$\circ$}};

\node at (-5.25, 0.5) {{$\varkappa(K_2)$}};

\node at (-4.2, 0.65) {\color{black}{$\bullet$}};

\node at (-3.2, 0.65) {\color{black}{$\bullet$}};
\node at (-3, 0.65) {\color{black}{$\bullet$}};

\node at (-2.2, 0.65) {\color{black}{$\bullet$}};
\node at (-2, 0.65) {\color{black}{$\bullet$}};
\node at (-1.8, 0.65) {\color{black}{$\bullet$}};

\node at (-1.2, 0.65) {\color{black}{$\bullet$}};
\node at (-1, 0.65) {\color{black}{$\bullet$}};
\node at (-0.8, 0.65) {\color{black}{$\bullet$}};
\node at (-1.2, 0.35) {\color{black}{$\bullet$}};

\node at (-0.2, 0.65) {\color{black}{$\bullet$}};
\node at (0, 0.65) {\color{black}{$\bullet$}};
\node at (0.2, 0.65) {\color{black}{$\bullet$}};
\node at (-0.2, 0.35) {\color{black}{$\bullet$}};
\node at (0, 0.35) {\color{black}{$\circ$}};

\node at (0.8, 0.65) {\color{black}{$\bullet$}};
\node at (1, 0.65) {\color{black}{$\bullet$}};
\node at (1.2, 0.65) {\color{black}{$\bullet$}};
\node at (0.8, 0.35) {\color{black}{$\bullet$}};
\node at (1, 0.35) {\color{black}{$\circ$}};

\node at (1.8, 0.65) {\color{black}{$\bullet$}};
\node at (2, 0.65) {\color{black}{$\bullet$}};
\node at (2.2, 0.65) {\color{black}{$\bullet$}};
\node at (1.8, 0.35) {\color{black}{$\circ$}};

\node at (2.8, 0.65) {\color{black}{$\bullet$}};
\node at (3, 0.65) {\color{black}{$\bullet$}};
\node at (3.2, 0.65) {\color{black}{$\circ$}};
\node at (2.8, 0.35) {\color{black}{$\circ$}};

\node at (3.8, 0.65) {\color{black}{$\bullet$}};
\node at (4, 0.65) {\color{black}{$\circ$}};

\node at (4.8, 0.65) {\color{black}{$\circ$}};

\node at (5.8, 0.65) {\color{black}{$\circ$}};

\draw[thin,->>] (-5.15,6.75) -- (-5.15,6.25);
\draw[thin,->>] (-5.15,5.75) -- (-5.15,5.25);
\draw[thin,left hook->] (-5.15,4.75) -- (-5.15,4.25);
\draw[thin,->>] (-5.15,3.75) -- (-5.15,3.25);
\draw[thin,left hook->] (-5.15,2.75) -- (-5.15,2.25);
\draw[thin,left hook->] (-5.15,1.75) -- (-5.15,1.25);
\end{tikzpicture}

\label{fig:Kappa}
\end{figure}

\section{L-space knots without thin surgeries} 

We finish this article by showing furthermore that the knots $K_1$ and $K_2$ have no thin surgeries at all, thereby proving Theorem~\ref{thm:counterexample_thin}. We again write $K$ for either $K_1$ or $K_2$. Since $K$ is hyperbolic, it follows from Thurston's hyperbolic Dehn surgery theorem~\cite{Th79}, that whenever the length $L(r)$ of a slope $r$ is sufficiently large, $K(r)$ is hyperbolic and its symmetry group injects into the symmetry group of $K$ (see, for example, \cite[Corollary 2.5]{BKM_QA}).\footnote{The \textit{symmetry group} of a $3$-manifold $M$ is defined to be its mapping class group. If $M$ is hyperbolic then the symmetry group of $M$ is isomorphic to its isometry group~\cite{Gabai_DiffMCG,Gabai_DiffMCG2}.} Since the symmetry group of $K$ is $\Z_2$ and contains a unique strong inversion, for all sufficiently large $r$ the manifold $K(r)$ admits a unique description as a double branched cover over $S^3$, namely the one with branching set $T(r)$ discussed in the previous section. Since the $T(r)$ are never thin by Proposition~\ref{prop:str_inv}, it follows that $r$ can be a thin slope for $K$ only if $r$ is an exceptional slope (i.e $K(r)$ is not a hyperbolic manifold) or $r$ is a \emph{symmetry-exceptional} slope for $K$ (i.e $K(r)$ is hyperbolic with a symmetry group larger than that of $K$).

\subsection{Exceptional slopes}
First, we will show that the exceptional slopes are not thin slopes. We need two preliminary results allowing us to deal with toroidal and Seifert fibered exceptional surgeries. 

\begin{lem}\label{lem:sfs_branching}
Let $M$ be small Seifert fibered L-space with infinite fundamental group. Then $M$ has a unique description as the double branched cover of a link in $S^3$. 
\end{lem}

\begin{proof}
It was shown in~\cite{Montesinos} that any small Seifert fibered space (without the L-space condition or the condition on the fundamental group) arises as the double branched cover of a Montesinos link.

To establish uniqueness, we suppose that $M$ is the double branched cover of a link $L$. Since $M$ has infinite fundamental group, Proposition 3.3 in~\cite{Mo17} implies that $L$ is either a Montesinos link or a link with Seifert fibered complement. The main result of~\cite{ADE_link} shows that if the double branched cover of a link $L$ with Seifert fibered complement is an L-space, then $L$ also admits a description as a Montesinos link. However, from the proof of Proposition 3.3 in~\cite{Mo17} (see Claim 3.6), it follows that the description of $M$ as the double branched cover of a Montesinos link is unique.  
\end{proof}

\begin{lem}\label{lem:dbcgivingunionofknotexteriors}
	Let $M$ be a $3$-manifold that has a JSJ decomposition into the exteriors of two non-trivial, non-satellite knots in $S^3$, each with unique strong inversions. If $M$ is the double branched cover of knots $L_1$ and $L_2$ in $S^3$, then $L_1$ and $L_2$ are mutants of each other.
    In particular, $L_1$ is thin if and only if $L_2$ is thin.
\end{lem}		

\begin{proof}
By assumption $M$ has a JSJ decomposition $M=M_1 \cup_T M_2$ where $M_i = S^3 \cut \nbhd(K_i)$ for some non-trivial knots $K_i$, $i=1,2$ whose complements are atoroidal. Under these hypotheses $T=M_1 \cap M_2$ is the unique essential torus in $M$ up to isotopy. 

Suppose $M$ is the double branched cover of a knot $L$ in $S^3$. 
From this double branched cover, $M$ inherits an involution $\tau$. Since $M$ is not a Seifert fibered space, the Equivariant Torus Theorem \cite[Corollary~4.6]{equivarianttorustheorem} shows $M$ contains an incompressible torus $T'$ such that either $\tau(T')=T'$ or $\tau(T') \cap T' = \emptyset$. Since $M$ contains a unique essential torus up to isotopy we may assume that $T'=T$. Furthermore, the uniqueness of $T$ implies that $\tau(T)=T$.  

Next, we show $T$ cannot be disjoint from the fixed set of $\tau$. Indeed, if this were the case, then $T$ would be the lift of a torus $T_0$ in $S^3$ disjoint from $L$. As a torus in $S^3$, $T_0$ must bound the exterior $E$ of a knot $K$ to one side and a solid torus $V$ to the other. Then, say, $M_1$ is the lift of $E$ and $M_2$ is the lift of $V$. Since the boundary of $M_2$ is incompressible in $M$, we see that $V$ must contain the branching set $L$. Since the boundary of $M_1$ is incompressible in $M$ and $M_1 \rightarrow E$ is a covering, $E$ is not a solid torus.
However~\cite[Theorem 3.4]{GAW-imbeddings} together with the resolution of the Poincar\'e Conjecture~\cite{perelman1,perelman2} shows that some $(2/n)$-surgery on $K$ must yield $\R P^3$.   Yet only the unknot admits a Dehn surgery to $\R P^3$ \cite{KMOS}, contradicting that $E$ is not a solid torus. This completes the argument showing that $T$ and the fixed set of $\tau$ are not disjoint.
 
We conclude that $T$ intersects the fixed set of $\tau$. Then $T$ is the lift of an essential Conway sphere that splits $L$ into two tangles whose double branched covers are the knot exteriors $M_1$ and $M_2$. Since $K_1$ and $K_2$ admit, by assumption, unique strong inversions, the branching sets for $M_1$ and $M_2$ are unique, it follows that any pair of branching sets for $M$ differ by mutation. 

Since Khovanov homology with $\Z_2$-coefficients is preserved under mutation~\cite{Mutation_KH}, it follows that all branching sets of $M$ have the same thinness status.
\end{proof}

\begin{lem}\label{lem:excepto9_30634}
	No exceptional slope of $K_1$ or $K_2$ is thin.
\end{lem}
			
\begin{proof}
From Dunfield's list~\cite{Du18}, we read off that the knot $K_1$ has exactly two exceptional slopes. These are the slope $r=18$, yielding the Seifert fibered space
	\begin{equation*}
		SFS[S2 :(2,1)(7,2)(8,-5)],
	\end{equation*}
and the slope $r=19$, yielding the graph manifold
	\begin{equation*}
		G= SFS[D :(2,1)(3,1)] \cup_m SFS[D :(2,1)(3,1)],\,\,m=\begin{pmatrix}
						2 & 1\\
						1 & 1
		\end{pmatrix},
	\end{equation*}
while the slope $r=14$ is the single exceptional slope of $K_2$ yielding the Seifert fibered space \[SFS [S2: (2,1) (8,3) (9,-7)].\footnote{Strictly speaking the orientations on these manifolds should be reversed, since the knots $K_1$ and $K_2$ are mirrors of the census knots $t09847$ and $o9\_30634$. That is, if $M$ is a manifold obtained by an exceptional surgery on one of these two census knots, then it is $-M$ that arises as an exceptional surgery on $K_1$ or $K_2$. However, this is irrelevant for the argument that the exceptional slopes are not thin.}\]  
That the two Seifert fibered spaces are not double branched covers of a thin link follows from Lemma~\ref{lem:sfs_branching}. Their unique descriptions as double branched covers are over links in $S^3$ necessarily have branching set of the form $T(r)$. However, the $T(r)$ are already known not to be thin by Proposition~\ref{prop:str_inv}.

The above graph manifold $G$ is the double branched
cover of the knot $T(19)$, which is not thin by Proposition~\ref{prop:str_inv}. As $H_1(G;\Z)$ has odd order every branching set for $G$ in $S^3$ is necessarily a knot \cite[Chapter~9]{lickorish}.
As $G$ is the result of gluing two trefoil knot exteriors along their boundaries, Lemma~\ref{lem:dbcgivingunionofknotexteriors} implies that no branching set for $G$ in $S^3$ can be thin.
\end{proof}

\begin{rem}
    In fact, it also follows that $G$ is the double branched cover of a unique link.
    Indeed, $G$ is the double branched cover of $T(19)=K12n120$. It is not hard to check that $K12n120$ has a unique Conway sphere and that any mutation along that Conway sphere preserves this knot.
\end{rem}

\subsection{Symmetry-exceptional slopes}
Finally, we analyze the symmetry-excep\-tion\-al slopes. That there are only finitely many such slopes follows from Thurston's hyperbolic Dehn filling theorem. Recent work of Futer--Purcell--Schleimer allows us to bound the length of such symmetry-exceptional slopes \cite{FPS19}.

\begin{lem}\label{lem:slope}
Let $K$ be a hyperbolic knot in $S^3$. Then any symmetry-exceptional slope $r$ of $K$ satisfies
\[
\widehat{L}(r)\leq \max\left\{10.1\,,\,\sqrt{\frac{2\pi}{\operatorname{sys}(K)}+58}\right\},
\]
where $\operatorname{sys}(K)$ is the length of the shortest geodesic of $S^3\setminus K$ and $\widehat{L}(r)$ denotes the normalized length of $r$.
\end{lem}
\begin{proof}
If the normalized length of $r$ exceeds the stated bound, then work of Futer--Purcell--Schleimer~\cite[Theorem 7.28]{FPS19} says that $K(r)$ is a hyperbolic manifold and the core $\gamma$ of the surgery torus is the shortest geodesic in $K(r)$, for details we refer to Theorem~2.4 in~\cite{BKM_QA}. Since every element in the symmetry group of $K(r)$ is isotopic to an isometry, we may assume that it maps $\gamma$ to itself and thus restricts to a symmetry of the knot complement of $K$. 
\end{proof}

Now let $M$ be a closed $3$-manifold with a smooth involution $\Phi\colon M\rightarrow M$. 
Then we can take the quotient under $\Phi$, i.e.\ we identify points $p$ and $\Phi(p)$. The quotient map $M\rightarrow M/\Phi$ will be a $2$-fold branched covering map with branching set $B_\Phi$.

\begin{lem}\label{lem:dbc_lemma}
Suppose $M$ is a closed hyperbolic $3$-manifold with two smooth involutions $\Phi_1,\Phi_2\colon M\rightarrow M$ that are isotopic as diffeomorphisms. Then $(M/\Phi_1,B_{\Phi_1})$ is diffeomorphic to $(M/\Phi_2,B_{\Phi_2})$.
\end{lem}

\begin{proof}
  When we see $M/\Phi_i$ as an orbifold $\mathcal{O}_i$, then the quotient map $M\rightarrow \mathcal O_i$ is a $2$-fold orbifold covering map. Since $M$ is hyperbolic, the orbifold universal covering of $\mathcal O_i$ is also hyperbolic, and thus $\mathcal O_i$ is irreducible and atoroidal. Thus the orbifold geometrization theorem~\cite[Corollary~1.2]{BLP_annals} 
  implies that $\mathcal O_i$ is hyperbolic and the deck transformation $\Phi_i$ is an isometry for some hyperbolic metric $g_i$ on $M$. Then Mostow rigidity implies that $\Phi_1$ and $\Phi_2$ are conjugated by an isometry. But this implies that $(M/\Phi_1,B_{\Phi_1})$ is diffeomorphic to $(M/\Phi_2,B_{\Phi_2})$.
\end{proof}

\begin{lem}\label{lem:symexcepto9_30634}
	No symmetry-exceptional slope of $K_1$ or $K_2$ is thin.
\end{lem}
			
\begin{proof}
The verified functions in SnapPy together with the bound from Lemma~\ref{lem:slope} show that the symmetry-exceptional slopes of $K_1$ are $16$, $20$, and $21$; and of $K_2$ are $12$, $15$, and $16$. For each symmetry-exceptional slope $r$, the hyperbolic manifold $K_i(r)$ has symmetry group $\Z_2^2$. 

First, we use SnapPy to check that any exceptional symmetry is orientation-preserving, and thus by Lemma~\ref{lem:dbc_lemma} each exceptional symmetry results in a unique description of $K(r)$ as double branched cover over a link in some quotient manifold. In order to verify that $r$ is not a thin slope, it is enough to demonstrate for each of these descriptions that the quotient manifold is not $S^3$. For each symmetry exceptional slope, these descriptions as double branched covers are recorded in Tables~\ref{tab:symexct09847} and~\ref{tab:symexco9_30634}. 
\end{proof}

\begin{notation}\label{not:table_notation}
Tables~\ref{tab:symexct09847} and~\ref{tab:symexco9_30634}
include surgery descriptions of branching sets for quotients of orientation preserving involutions of manifolds. When the branching set is not contained in $S^3$, these are presented by link in the Hoste--Thistlethwaite tabulation as recorded in SnapPy with a list of surgery coefficients. These links come with an ordering on the components, and the surgery information is presented in a list-without-commas of an ordered pair $(p_i, q_i)$ indicating $p_i/q_i$ surgery on the $i$th component. The ordered pair $(0,0)$ is taken to mean that no surgery is performed on the component and it forms part of the branching set. For example, $L13n9355(0,0)(-5,2)(0,0)$ describes a link of two components arising from the first and third component of $L11n350$ in the manifold obtained by $-5/2$ surgery on the second component.
\end{notation}

\begin{rem}
    We conclude the section by explaining the strategy used to find the descriptions of the symmetry-exceptional surgeries as branched double covers. 
    For any $r$ we have a description of $K(r)$ as double branched covers of the $T(r)$. 
    The remaining descriptions of the $K(r)$ were found by a brute force search. Given a link $L$ in the HTW table, we performed Dehn surgery on some subset of its components, leaving one or two components unfilled. This gives a pair $(M,B)$ where $M$ is the manifold obtained from the surgered components and $B$ is the link in $M$ formed by the unfilled components of $L$. Then we take all double covers of $M$ branched along $B$\footnote{Note that in this setting, there might be more than one or no double branched cover of $B$ depending on the algebraic topology of the exterior of $B$.} and check if it is diffeomorphic to $K(r)$. This identified for any orientation-preserving exceptional symmetry a link in a lens space whose double branched covering is diffeomorphic to a symmetry-exceptional filling of $K_1$ or $K_2$. 
\end{rem}

\begin{table}[htbp]
{\small
\caption{Branching sets of symmetry-exceptional slopes of $K_1$.}
\label{tab:symexct09847}
\begin{tabular}{@{}
>{\raggedright}p{0.1\linewidth}
>{\raggedright}p{0.15\linewidth}	
>{\raggedright}p{0.21\linewidth}
>{\raggedright\arraybackslash}p{0.31\linewidth}@{}}
\toprule
slope  & symmetry &  quotient manifold& branching set \\
\midrule

\multirow{3}{*}{$16$}&
\multirow{3}{*}{\begin{tabular}{@{}>{\raggedright}p{0.8\linewidth}>{\raggedleft}p{0.1\linewidth}}		$\Z_2^2$ & {\Vast\{} \end{tabular}}
&$S^3$& $L14n31171=T(16)$  \\ 
&&$L(2,1)$&$L12n1082(-2,1)(0,0)$\\ 
& 	&$-L(8,3)$&$L10n92(0,0)(-7,4)(-8,3)$\\ 
\midrule
\multirow{3}{*}{$20$}&
\multirow{3}{*}{\begin{tabular}{@{}>{\raggedright}p{0.8\linewidth}>{\raggedleft}p{0.1\linewidth}}		$\Z_2^2$ & {\Vast\{} \end{tabular}}
&$S^3$& $L14n24290=T(20)$  \\ 
& 	 &$L(5,2)$&$L13n9355(0,0)(-5,2)(0,0)$\\ 
& 	&$L(2,1)$&$L13n7290(0,0)(-2,1)(0,0)$\\
\midrule
\multirow{3}{*}{$21$}&
\multirow{3}{*}{\begin{tabular}{@{}>{\raggedright}p{0.8\linewidth}>{\raggedleft}p{0.1\linewidth}}		$\Z_2^2$ & {\Vast\{} \end{tabular}}
&$S^3$&$K15n88871=T(21)$\\
& &$L(3,1)$& $L13n5917(-3, 1)(0,0)$  \\ 
& 	&$-L(7,5)$&$L11n294(-2,3)(-5,1)(0,0)$\\
\bottomrule
\end{tabular}
}
\end{table}
			
\begin{table}[htbp]
{\small
\caption{Branching sets of symmetry-exceptional slopes of~$K_2$.}
\label{tab:symexco9_30634}
\begin{tabular}{@{}
>{\raggedright}p{0.09\linewidth}
>{\raggedright}p{0.14\linewidth}	
>{\raggedright}p{0.20\linewidth}
>{\raggedright\arraybackslash}p{0.35\linewidth}@{}}
\toprule
slope  & symmetry &  quotient manifold& branching set \\
\midrule
\multirow{3}{*}{$12$}&
\multirow{3}{*}{\begin{tabular}{@{}>{\raggedright}p{0.8\linewidth}>{\raggedleft}p{0.1\linewidth}}		$\Z_2^2$ & {\Vast\{} \end{tabular}}
& $S^3$& $T(12)$  \\ 
& 	 &$L(3,2)$&$L12n2163(0,0)(-2,1)(-2,1)(0,0)$\\ 
& 	&$L(2,1)$&$L13n9873(0,0)(-2,1)(0,0)$\\
\midrule
\multirow{3}{*}{$15$}&
\multirow{3}{*}{\begin{tabular}{@{}>{\raggedright}p{0.8\linewidth}>{\raggedleft}p{0.1\linewidth}}		$\Z_2^2$ & {\Vast\{} \end{tabular}}
& $S^3$& $K13n2148=T(15)$  \\ 
& 	 &$L(5,1)$&$L10n45(-5,1)(0,0)$\\ 
& 	&$L(3,1)$&$L11n141(-3,1)(0,0)$\\
\midrule
\multirow{3}{*}{$16$}&
\multirow{3}{*}{\begin{tabular}{@{}>{\raggedright}p{0.8\linewidth}>{\raggedleft}p{0.1\linewidth}}		$\Z_2^2$ & {\Vast\{} \end{tabular}}
&$S^3$&$T(16)$\\
& &$L(8,3)$& $L12n1012(-8, 3)(0,0)$  \\ 
& 	&$L(2,1)$&$L14n23888(-2,1)(0,0)$\\ 
\bottomrule
\end{tabular}
}
\end{table}
			
\let\MRhref\undefined
\bibliographystyle{hamsalpha}
\bibliography{altsurg.bib}

\end{document}